\documentclass[11pt]{amsart}
\usepackage[utf8]{inputenc} % PERMITE USAR DIGITA?O COMO NO WORD
\usepackage[pctex32]{graphics}
\usepackage{amsfonts}
\usepackage{amssymb,amscd,latexsym}
\usepackage{amsmath}
\usepackage{epsfig}
\usepackage{color}

\newtheorem{thm}{Theorem}[section]
\newtheorem{cor}[thm]{Corollary}
\newtheorem{prop}[thm]{Proposition}

\newtheorem{defin}[thm]{Definition}

\newtheorem{lema}[thm]{Lemma}
\newtheorem{rmk}[thm]{Remark}

\newtheorem{conjecture}[thm]{Conjecture}

\newtheorem{ex}[thm]{Example}

\newcommand{\codim}{\operatorname{codim}}

\newcommand{\chr}{\operatorname{char}}

\def\R{\mathbb{R}}
\def\C{\mathbb{C}}

\def\P{\mathbb{P}}
\def\K{\mathbb{K}}
\def\rk{\operatorname{rk}}

\def\Hess{\operatorname{Hess}}
\def\hess{\operatorname{hess}}

\newcommand{\Hilb}{\operatorname{Hilb}}

\newcommand{\ann}{\operatorname{Ann}}
\newcommand{\Hom}{\operatorname{Hom}}

\newcommand{\rank}{\operatorname{rank}}

\newcommand{\ba}{\mathcal{B}}
\newcommand{\rw}{\Rightarrow}

\begin{document}

\title{On mixed Hessians and the Lefschetz properties}

\author[R. Gondim]{Rodrigo Gondim*}
\address{Universidade Federal Rural de Pernambuco, av. Don Manoel de Medeiros s/n, Dois Irmãos - Recife - PE
52171-900, Brasil}
\email{rodrigo.gondim@ufrpe.br}
\author[G. Zappalà]{Giuseppe Zappalà}
\address{Dipartimento di Matematica e Informatica, Universit\` a degli Studi di Catania, Viale A. Doria 5, 95125 Catania, Italy}
\email{zappalag@dmi.unict.it}

\begin{abstract} We introduce a new type of Hessian matrix, that we call Mixed Hessian. The mixed Hessian is used to compute the rank of a multiplication map by a power of a linear form in a standard graded Artinian Gorenstein algebra.
In particular we recover the main result of \cite{MW} for identifying Strong Lefschetz elements, generalizing it also for Weak Lefschetz elements. This criterion is also used to give a new proof that Boolean algebras have the Strong Lefschetz Property (SLP).
We also construct new examples of Artinian Gorenstein algebras presented by quadrics that does not satisfy the Weak Lefschetz Property (WLP); we construct minimal examples of such algebras and we give bounds, depending on the degree, for their 
existence. 
Artinian Gorenstein algebras presented by quadrics were conjectured to satisfy WLP in \cite{MN1,MN2}, and in a previous paper we construct the first counter-examples (see \cite{GZ}).

\end{abstract}

\thanks{*Partially supported  by the CAPES postdoctoral fellowship, Proc. BEX 2036/14-2}

\maketitle

\section*{Introduction}

The Hessian matrix of a form is the matrix of its second derivatives and its Hessian is the determinant of this matrix. The first instance of such object goes back to the seminal paper of Gauss \cite{Ga}.
In this context the Hessian describes curvature for surfaces given by an implicit function, see also Segre \cite{Se} for the $n$-dimensional analog.
Complete hypersurfaces with zero Gaussian curvature are also called developable.
 We recall that for $X=V(f)\subset \P_{\K}^N$, a hypersurface defined over $\K=\R,\C$, we get $\hess_f = 0 \pmod{f}$ if and only if the hypersurface is developable, that is, the Gauss map is degenerated. 
In $\P^3$ only cones and the tangent surface of a curve are developable. While the cones have $\hess_f=0$ the tangent surfaces have $\hess_f \neq 0$ (see \cite[Chapter 7]{Ru}).

Hesse claimed in \cite{He} that for arbitrary $N$, $\hess_f=0$ if and only if $X=V(f)\subset \P^N$ is a cone. Gordan and Noether in the fundamental paper \cite{GN} showed that Hesse's claim is true for $N \leq 3$ and 
they produced series of counterexamples for $N \geq 4$.  Moreover, these counterexamples can be characterized as the only hypersurfaces in $\P^4$ with vanishing hessian which are not cones. 
A modern proof of this fact can be found in \cite{GR} while a very detailed account on the subject appears in \cite[Chapter 7]{Ru}. The so called Gordan-Noether Theory is also 
treated in very different aspects in \cite{CRS, Lo, GR, GRu, Wa2, Ru}. \par

Hessians of higher degree were introduced in \cite{MW} and used to control the so called 
Strong Lefschetz property (SLP). This property for a Graded Artinian Gorenstein algebra was inspired by the Hard Lefschetz Theorem on the cohomology of 
smooth projective complex varieties. In this paper we introduce the mixed Hessians, that generalize the Hessians of higher order, providing a generalization 
of the criterion for Strong Lefschetz elements also for Weak Lefschetz elements (see Theorem \ref{thm:generalization} and Corollary \ref{cor:hessiancriteria}). \par

The Lefschetz properties have attracted a great deal of attention over the years, since they are {\it phenomena} connected with Commutative Algebra, Algebraic and Tropical Geometry and Combinatorics, 
see \cite{St,St2,HMMNWW,HMNW, Go, GZ}. 
The first result in the area, proved by Stanley \cite{St} and independently by Watanabe, asserts that a complete intersection of monomials have the SLP. Here we (re)prove a special case of this result for quadratic complete intersections of monomials, 
also called Boolean algebras. \par

A standard graded $\K$ algebra is said to be presented by quadrics if it is isomorphic to the quotient of a polynomial ring over $\K$ by a homogeneous ideal 
generated by quadratic forms. These algebras are related with Koszul algebras and Gr\"obner basis, see for example \cite{Co}. 
In \cite{GZ} we disprove a conjecture posed in \cite{MN2} that Artinian Gorenstein algebras presented by quadrics have the WLP. 
Here we study in more details the family introduced in \cite{GZ} to give minimal examples for those algebras failing WLP. \par

We now describe the contents of the paper in more detail. In the first section we recall the basic definitions and constructions of standard graded Artinian Gorenstein algebras, 
and we recall a combinatoric construction introduced in \cite{GZ}. \par

In the second section we introduce the mixed Hessians and prove the main result, Theorem \ref{thm:generalization}, a generalization of the Hessian criterion to mixed Hessians, see also Corollary \ref{cor:hessiancriteria}.
In the third section we prove an inductive construction (see Proposition \ref{prop:inductiveconstruction}) whose Corollary is the very well known fact that Boolean algebras have the SLP (see Corollary \ref{cor:booleanalgebras}). \par

The next section is devoted to recall a combinatorial construction introduced in \cite{GZ}, we associate a pure simplicial complex to a standard graded Artinian Gorenstein algebra. A special family called Turan algebras have been used in \cite{GZ} to 
produce counterexamples to the conjecture posed in \cite{MN1,MN2}. The conjecture was that Artinian Gorenstein presented by quadrics have WLP. \par  

In the last section we deal with algebras presented by quadrics of minimal codimension failing the WLP. For degree $d=3$ we find the minimal example in codimension $8$ (see Example \ref{ex:semplicissimo}). We also classify algebras associated to graphs 
with respect to WLP (see Proposition \ref{prop:maincubic}). Applying the inductive construction we get a lower bound for the 
codimension of algebras of odd degree to fail the WLP (see Corollary \ref{cor:existemimpares}); this bound is relatively sharp. For even degrees we also give a bound for the failure of the WLP, Corollary \ref{cor:existempares}.

\section{Artinian Gorenstein algebras and the Lefschetz properties}

\subsection{Lefschetz properties}

Let $\K$ be an infinite field and  $R=\K[x_1,\ldots,x_n]$ be the polynomial ring in $n$ indeterminates. 

%\begin{defin}\rm
% Let $A$ be a standard graded $\K$-algebra.  We say that $A$ is presented by quadrics if $A\simeq R/I$, where $R=\K[x_1,\ldots,x_n]$ and the homogeneous ideal $I$ 
% has a set of generators consisting of quadratic forms. 
%\end{defin}

Let $A=R/I$ be an Artinian standard graded $R$-algebra, then $A$ has a decomposition $A=\displaystyle \bigoplus_{i=0}^dA_i,$ as a sum of finite dimensional $\K$-vector spaces with $A_d\ne 0$.

Let $A=R/I$ be an Artinian standard graded $R$-algebra. 
A form $F\in R_d$ induces a $\K$-vector spaces map $\mu_{i,F}:A_i\to A_{i+d},$ defined by $\mu_{i,F}(\alpha)=F\alpha,$ for every $\alpha\in A_i.$

\begin{defin}\rm
We say that $A$ has the {\em Strong Lefschetz property} (in short SLP) if there exists a linear form $L\in R_1$ such that
$\rk\mu_{i,L^k}=\min\{\dim_{\K} A_i,\dim_{\K} A_{i+k}\},$ for every $i,k.$ 
%$\bullet L:A_i \to A_{i+1}$ has maximal rank for every $i.$ 
\end{defin}

\begin{defin}\rm
We say that $A$ has the {\em Weak Lefschetz property} (in short WLP) if there exists a linear form $L\in R_1$ such that
$\rk\mu_{i,L}=\min\{\dim_{\K} A_i,\dim_{\K} A_{i+1}\},$ for every $i.$ 
%$\bullet L:A_i \to A_{i+1}$ has maximal rank for every $i.$ 
\end{defin}

%\begin{Definition}\rm
%We say that $A$ has the {\em Strong Lefschetz property} (in short SLP) if there exists a linear form $L\in R_1$ such that 
%$\rk\mu_{i,L^d}=\min\{\dim_{\K} A_i,\dim_{\K} A_{i+d}\},$ for every $i$ and $d.$ 
%\end{Definition}

\begin{defin}\rm 
 Let $R=\K[x_1,\ldots,x_n]$ and $A=R/I$ be an Artinian standard graded $R$-algebra, with $I_1=0$. The integer $n$ is said to be the codimension of $A$. 
 If $A_d \neq 0$ and $A_i=0 $ for all $i > d$, then $d$ is called the socle degree of $A$. The Hilbert vector of $A$ is $h_A=\Hilb(A)=(1,h_1,h_2,\ldots,h_d)$, where 
 $h_k = \dim A_k$. We say that $h_A$ is unimodal if there exists $k$ such that $1\leq h_1\leq \ldots \leq h_k \geq h_{k+1}\geq h_d$.
\end{defin}

\begin{rmk}\rm \label{rmk:simplifica}
We recall that an Artinian algebra $A=\displaystyle \bigoplus_{i=0}^dA_i,$ $A_d\ne 0,$ is a {\em Gorenstein algebra} if and only if $\dim_{\K} A_d=1$ and the bilinear pairing 
$$A_i\times A_{d-i}\to A_d$$ induced by the multiplication is non-degenerated for $0\le i\le d$. 
So we have an isomorphism $A_{i}\simeq\Hom_{\K}(A_{d-i},A_d)$ for $i=0,\ldots,d$.
%In this case, after choose an isomorphism $A_d \simeq \K$, we have $A_{d-i} \simeq A_i^*$ for $i=0,\ldots,d$. 
In particular, $\dim_{\K} A_i=\dim_{\K} A_{d-i},$ for $i=0,\ldots,d$.
Moreover, for every $L\in R_1,$ $\rank\mu_{i,L}=\rank\mu_{d-i-1,L},$ for $0\le i\le d.$
\par
Since $A$ is generated in degree $0$ as a $R$-module,  if $\mu_{i,L}$ is surjective, then $\mu_{j,L}$ is surjective for every $j\ge i.$ 
Therefore, if $A$ is a Gorenstein Artinian algebra, if $\mu_{i,L}$ is injective, then $\mu_{j,L}$ is injective for every $j\le i.$ 
Of course SLP implies WLP. Notice also that the WLP implies the unimodality of the Hilbert vector of $A$. Unimodality in the Gorenstein case implies that $\dim A_{k-1} \leq \dim A_k$ for all $k \leq \frac{d}{2}$.
The converse of these implications are not true, (see \cite{Go}).
\end{rmk}

%\par\rosso{inserire l'osservazione che basta un sol posto per WLP}
\par

\subsection{Macaulay-Matlis duality}

From now on we assume that $\chr\K=0.$ Let us regard the polynomial algebra $R$ as a module over the algebra $Q=\K[X_1, . . . ,X_n]$ via the identification $X_i = \partial/\partial x_i.$ If $f\in R$ we set
 $$\ann_Q(f)=\{p(X_1,\ldots,X_n)\in Q\mid p(\partial/\partial x_1,\ldots,\partial/\partial x_n)f=0\}.$$
 It is well known that $A=Q/I$ is a Gorenstein standard graded Artinian algebra if and only if there exists a form $f\in R$ such that $I=\ann_Q(f)$ (for more details see, for instance, \cite{MW}).

In the sequel we always assume that $\operatorname{char}(\K)=0$, $A=Q/I$, $I=\ann_Q(f)$ and $I_1=0$. 

When we deal with standard bigraded Artinian Gorenstein algebras $A=\displaystyle \bigoplus_{i=0}^d A_i,$ $A_d\ne 0$, 
with $A_k=\displaystyle \bigoplus_{i=0}^k A_{(i,k-1)}$, $A_{(d_1,d_2)}\ne 0$ for some $d_1,d_2$ such that $d_1+d_2=d$, we call $(d_1,d_2)$ the socle bidegree of $A$. Since $A^*_k \simeq A_{d-k} $ and since duality is compatible with 
direct sum, we get $A_{(i,j)}^* \simeq A_{(d_1-i,d_2-j)}$. \\
In this case given a presentation of $A = Q/\ann_f$ with $R=\K[x,u]$ and $Q = \K[X,U]$ standard bigraded, we get $I=\ann_f$ a bihomogeneous ideal. It is easy to see that the Macaulay dual of the defining ideal is $f \in R_{(d_1,d_2)}$ a bihomogeneous polynomial of total degree $d=d_1+d_2$.

%\rosso{da scrivere}

\begin{defin}\rm \label{defin:bigraded1}
 With the previous  notation, all bihomogeneous polynomials of bidegree $(1,d-1)$ can be written in the form
 $$f= x_1g_1+\ldots+x_ng_n,$$
 where $g_i \in \K[u_1,\ldots,u_m]_{d-1}$. We say that $f$ is of {\em  square-free monomial type} if all 
 $g_i$ are square free monomials. The associated algebra, $A = Q/\ann_Q(f)$, is bigraded, has socle bidegree $(1,d-1)$ and we assume that $I_1=0,$ so $\codim A = m+n$.
\end{defin}

\section{Mixed Hessians and dual mixed Hessians}
Let $R=K[x_1,\ldots,x_n]$ and $Q=K[\partial/\partial x_1,\ldots,\partial/\partial x_n].$ Let $f\in R_d.$
Let $A=Q/\ann(f)$ be a standard graded Artinian Gorenstein $K$-algebra, 
   $$A=A_0\oplus\ldots\oplus A_d,\,\,\,\dim_K A_d=1.$$ 
Let $k\le l$ two integers, take $L\in A_1$ and let us consider the $K$-vector spaces map
    $$\mu_L: A_k\to A_l,\,\,\mu_{L}(\alpha)=L^{l-k}\alpha.$$
Let $\ba_k=(\alpha_1,\ldots,\alpha_r)$ be a $K$-linear basis of $A_k$ and 
$\ba_l=(\beta_1,\ldots,\beta_s)$ be a $K$-linear basis of $A_l.$

\begin{defin}\rm
We call mixed Hessian of $f$ of mixed order $(k,l)$ with respect to the basis $\ba_k$ and $\ba_l$ the matrix: 
  $$\Hess_f^{(k,l)}:=[ \alpha_i\beta_j(f)]$$
Moreover, we define $\Hess_f^k=\Hess_f^{(k,k)}$, $\hess_f^k = \det(\Hess_f^k)$ and $\hess_f=\hess_f^1$.
  \end{defin}

Now we consider the unique generator $\vartheta\in A_d,$ such that $\vartheta(f)=1.$
So we can define the dual basis in $\Hom_K(A_l,A_d),$ $\ba^*_l=(\beta^*_1,\ldots,\beta^*_s),$ 
in the following way
  $$\beta^*_i(\beta_j)=\delta_{ij}\vartheta.$$
Since $A$ is Gorenstein, the multiplication induces a non-degenerate bilinear map $A_l\times A_{d-l}\to A_d,$ so
we have an isomorphism $\varphi:A_{d-l}\to\Hom_K(A_l,A_d),$ defined by $\varphi(\gamma)(\beta)=\gamma\beta.$ 
In particular we have $\varphi^{-1}(\beta^*_i)=\vartheta/\beta_i\in A_{d-l}.$

\begin{defin}\rm
We call dual mixed Hessian matrix the matrix 
  $$\Hess_f^{(l^*,k)}:=[(\vartheta/\beta_i)\alpha_j(f)]$$
\end{defin}
Note that $\Hess_f^{(l^*,k)} \in (R_{l-k})^{s,r}.$ 

\begin{rmk}\label{rmk:rem1}\rm First of all, since we are interested only in the rank of such matrices, the dependence on the basis is not relevant. \par
Therefore, it is easy to see that $\rk \Hess_f^{(l^*,k)}=\rk \Hess_f^{(n-l,k)}$.\par
 We observe that, under the natural assumption that $\ann_Q(f)_1 \neq 0$, the notation $\hess_f$ is consistent with the classical definition of Hessian, by taking $B_1=\{X_1,\ldots,X_n\}$, 
 the standard basis of the embedding.\par
 Moreover, the notation is also compatible with the Definition of higher order Hessians given in \cite{MW}.\\
 If $A$ is bigraded, and if $B_k=\{\alpha_1,\ldots,\alpha_s\}$ and $B_l=\{\beta_1,\ldots,\beta_t\}$ are bases of the $\K$-vector spaces $A_{(k,l)}$ and $A_{(k',l')}$ respectively,
we can also define $\Hess_f^{((k,l),(k',l'))}=(\alpha_i(\beta_j(f)))_{s \times t}$.
\end{rmk}

\par
If $L=a_1\partial/\partial x_1+\ldots+a_n\partial/\partial x_n,$ 
we set $L^{\perp}=(a_1,\ldots,a_n).$ 

\begin{thm}\label{thm:generalization}
With the previous notation, let $M$ be the matrix associated to the map $\mu_L:A_k \to A_l$ with respect 
to the bases $\ba_k$ and $\ba_l.$ Then
    $$M=(l-k)!\Hess_f^{(l^*,k)}(L^{\perp}).$$
\end{thm}
\begin{proof} First of all note that if $g\in R_h$ then $L^h(g)=h!g(L^{\perp}).$\\
Let $M=(b_{ij}).$ Then 
  $$L^{l-k}\alpha_j=\sum_{h=1}^s b_{hj}\beta_h.$$
Consequently 
 $$\beta^*_i(L^{l-k}\alpha_j)=\sum_{h=1}^s b_{hj}\beta^*_i(\beta_h) \rw 
(\vartheta/\beta_i)L^{l-k}\alpha_j=b_{ij}\vartheta \rw L^{l-k}(\vartheta/\beta_i)\alpha_j=b_{ij}\vartheta.$$
Now we evaluate in $F$ 
 $$L^{l-k}(\vartheta/\beta_i)\alpha_j(f)=b_{ij}\vartheta(f) \rw
    (l-k)!(\vartheta/\beta_i)\alpha_j(f)(L^{\perp})=b_{ij}.$$
\end{proof}

 The previous results give us a generalization of \cite[Theorem 4]{Wa1} and \cite[Theorem 3.1]{MW}.

\begin{cor}{\bf (Hessian criteria for Strong and Weak Lefschetz elements)}\label{cor:hessiancriteria}

Let $A = Q/\operatorname{Ann}_Q(f)$ be a standard graded Artinian Gorenstein algebra of codimension $r$ and socle degree $d$ and let $L = a_1x_1+\ldots+a_rx_r\in A_1$, such that $f(a_1,\ldots,a_r)\neq 0$. 
The map $\mu_{L^{l-k}}: A_k \to A_l$, for $k < l \leq \frac{d}{2}$, has maximal rank if and only if the (mixed) Hessian matrix $\Hess_f^{(k,d-l)}(a_1,\ldots,a_r)$ has 
maximal rank. 
In particular, we get the following:
\begin{enumerate}
 \item {\bf (Strong Lefschetz Hessian criterion, \cite{Wa1}, \cite{MW})} $L$ is a strong Lefschetz element of $A$ if and only if 
$\hess^k_f(a_1,\ldots, a_r)\neq 0$ for all $k=1,\ldots, [d/2]$.
\item {\bf (Weak Lefschetz Hessian criterion)} $L \in A_1$ is a weak Lefschetz element of $A$ if and only if either 
$d=2q+1$ is odd and $\hess^q_f(a_1,\ldots, a_r)\neq 0$ or $d=2q$ is even and $\Hess^{(q-1,q)}_f(a_1,\ldots, a_r)$ has maximal rank.
\end{enumerate}
\end{cor}

\begin{proof}
 %\rosso{da scrivere}

Let $\mu:A_k \to A_l$ be the map defined by the multiplication by $L^{l-k}.$ By Theorem \ref{thm:generalization}, 
$$\rk\mu=\rk\Hess_f^{(l^*,k)}(L^{\perp})=\rk\Hess_f^{(l^*,k)}(a_1,\ldots,a_r)=\rk\Hess_f^{(k,d-l)}(a_1,\ldots,a_r),$$ (see also Remark \ref{rmk:rem1}).
\par
The other claims are a direct consequence of it.

%For finite dimensional $\K$-vector spaces $\mathbb{V},\mathbb{W}$ we have the isomorphism: 
%$$\Hom_{\K}(\mathbb{V},\mathbb{W})\simeq \Hom_{\K}(\mathbb{V} \otimes \Hom_{\K}(\mathbb{W},\K),\K).$$ Therefore, to 
%any $\K$-linear map $T:\mathbb{V} \to \mathbb{W}$ we can associate, in a unique natural way, a bilinear form $B_T: %\mathbb{V} \times \mathbb{W}^* \to \K$. Furthermore,  
% $T$ has maximal rank if, and only if, $B_T$ has maximal rank. In our case of interest, the multiplication of the %algebra induces the dual pairing, $A_l \times A_{d-l} \to A_d$  
 %Hence, the linear map $\bullet L^{l-k} :A_k \to A_l $ induces a bilinear map 
 %$A_k \times A_{d-l} \to A_d$ given by $(\alpha,\gamma)\mapsto \gamma L^{l-k}\alpha$. If $\alpha_1,\ldots,\alpha_s \in %A_k$ and $\gamma_1,\ldots,\gamma_t \in A_{d-l}$ are $\K$-linear bases, 
% then the matrix of $B_T$ with respect to these bases is given by $[B_T]=(\gamma_i.L^{l-k}.\alpha_j(f))_{t \times %s}=(L^{l-k}(\gamma_i\alpha_j(f)))_{t \times s}$. On the other side, $\Hess_f^{(k,d-l)}=[\gamma_j\alpha_i(f)]$. We set
 % $g_{ij}:=\gamma_j\alpha_i(f) \in R_{l-k}$. Since $L^{l-k}g_{ij}=(l-k)!g_{ij}(\underline{a})$, the result follows. 
	%The others claim follows from Remark \ref{rmk:simplifica}.

\end{proof}

%\section{Algebras failing the WLP}

%In this section we give a new proof of a result that can be found in \cite{Go}, showing the existence of standard graded Artinian Gorenstein algebras $A$ of codimension $N$ and socle degree $d$ failing WLP for all pairs $(N,d)$ such that 
%$d \geq 3$ and $N \geq 4$ except $(4,3), (4,4), (4,6), (5,4)$. 

\section{An inductive construction}

In this section we want to study the relations between the algebras $A=Q/\ann_f$ and $\tilde{A}=\tilde{Q}/\ann_{\tilde{f}}$ with $f \in 
 R = \K[x_1,\ldots,x_r]$ and $\tilde{f} = uf \in \tilde{R}  = \K[x_1,\ldots,x_r,u]$. As a Corollary we prove that Boolean algebras have the SLP. 
 This result have been proved in a number of diferent ways, it was the genesis of the area with the work of R. Stanley and J. Watanabe.

\begin{lema}\label{lema:ideal}
 Let $f \in 
 R = \K[x_1,\ldots,x_r]$ be a homogeneous polynomial of degree $d$ and let $\tilde{f} = uf \in \tilde{R}  = \K[x_1,\ldots,x_r,u]$. Let $Q$ and $\tilde{Q}$ be the rings of differential operators associated to 
 $R$ and $\tilde{R}$ respectively. Then
 $$\ann_{\tilde{Q}}(\tilde{f}) = \ann_Q(f)\tilde{Q}+U^2\tilde{Q} \subset \tilde{Q}.$$
 In particular, if $A=Q/\ann_Q(f)$ is presented by quadrics, then $\tilde{A}=\tilde{Q}/\ann_{\tilde{Q}}(\tilde{f})$ is also presented by quadrics.
 \end{lema}

\begin{proof}
 It is easy to see that if $\alpha \in \ann_Q(f)$, then $\alpha \in \ann_{\tilde{Q}}(\tilde{f})$, and also $U^2 \in \ann_{\tilde{Q}}(\tilde{f})$, hence 
 $I=\ann_Q(f)\tilde{Q}+U^2\tilde{Q} \subset \ann_{\tilde{Q}}(\tilde{f})$. To prove the equality, let $\overline{\alpha} \in \ann_{\tilde{Q}}(\tilde{f})/I$, then we can write:
 $$\overline{\alpha}=\overline{\beta}+U\overline{\gamma}.$$
 Where $\beta, \gamma \in Q/I \subset \tilde{Q}/I$. Therefore:
 $$\alpha(\tilde{f})=\beta(f)+U\gamma(f)=0.$$
 Which give us $\overline{\beta}=\overline{\gamma}=0$, hence $\overline{\alpha}=0$ and the result follows.
\end{proof}

\begin{lema}\label{lema:algebra}
 With the previous notation we have the following decomposition as $\K$ vector spaces: 
 $$\tilde{A}_k=A_k \oplus A_{k-1}U.$$
 \end{lema}

\begin{proof}
 Let $\{\beta_1,\ldots,\beta_s\} \subset A_k$ be a $\K$-basis of $A_k$ and let $\{\gamma_1,\ldots,\gamma_l\} \subset A_{k-1}$ be a $\K$-basis of $A_{k-1}$. 
 We claim that $\{\beta_1,\ldots,\beta_s,U\gamma_1,\ldots,U\gamma_l\} \subset \tilde{A}_k$ is a $\K$-basis of $\tilde{A}_k$. 
 \begin{enumerate}
  \item[(i)] Linear independence. Suppose that $$b_1\beta_1+\ldots+b_s\beta_s+c_1U\gamma_1+\ldots+c_lU\gamma_l=0.$$  
  Hence, by Lemma \ref{lema:ideal} $b_1\beta_1+\ldots+b_s\beta_s=0$ implying $b_1=\ldots=b_s=0$, in the same way $c_1U\gamma_1+\ldots+c_lU\gamma_l=0$ implying $c_1=\ldots=c_l=0$.
  \item[(ii)] Spanning. Let $\alpha \in \tilde{A}_k$, by Lemma \ref{lema:ideal}, $\alpha= \beta +U\gamma$, with $\beta \in A_k$ and $\gamma \in A_{k-1}$. Therefore 
  $\beta = b_1\beta_1+\ldots+b_s\beta_s$ and $\gamma = c_1U\gamma_1+\ldots+c_lU\gamma_l$ since $\{\beta_1,\ldots,\beta_s\} $ is a $\K$-basis of $A_k$ and 
  $\{\gamma_1,\ldots,\gamma_l\} $ is a $\K$-basis of $A_{k-1}$.
 \end{enumerate}

\end{proof}

\begin{prop}\label{prop:inductiveconstruction} With the same notation, if $A$ has the SLP, then $\tilde{A}$ has the SLP.
\end{prop}

\begin{proof}
 By Lemma \ref{lema:ideal} and Lemma \ref{lema:algebra}, we get:
 
 $$ \Hess^k_{\tilde{f}} = \left[\begin{array}{cc}
                               0  & \Hess^{(k-1,k)}_f\\
                               \Hess^{(k,k-1)}_f & u \Hess^k_f
                              \end{array}\right] $$
By hypothesis and by Corollary \ref{cor:hessiancriteria}, $\hess_f^k \neq 0$, hence one can apply the determinant of block matrix to get:
$$ \hess^k_{\tilde{f}} = u^s\hess_f^k\det[0-\Hess^{(k-1,k)}_f(u \Hess^k_f)^{-1}\Hess^{(k,k-1)}_f]$$
Multiplying by $u$ we get:
$$ \hess^k_{\tilde{f}} = \hess_f^k\det[-\Hess^{(k-1,k)}_f(\Hess^k_f)^{-1}\Hess^{(k,k-1)}_f]$$
By Theorem \ref{thm:generalization} we can interpret the multiplication $[\Hess^{(k-1,k)}_f(\Hess^k_f)^{-1}\Hess^{(k,k-1)}_f]$, up to a scalar multiple, as a composition of multiplication maps by a 
general linear form $L \in A_1$ in the following way:
$$\begin{array}{ccccccc}
   A_{k-1} & \to & A_{d-k} & \to & A_k & \to A_{d-k+1}\\
   \alpha & \mapsto & L^{d-2k+1}\alpha & \mapsto & L\alpha & \mapsto L^{d-2k+2}\alpha 
  \end{array} $$
In fact,      $\Hess^{(k,k-1)}_f(L^{\perp})$ is the matrix of the map $\mu_{L^{d-2k+1}}:A_{k-1} \to A_{d-k}$, $(\Hess^k_f)^{-1}(L^{\perp})$ is the inverse of the matrix of the map
$\mu_{L^{d-2k}} :A_k \to A_{d-k}$ and      $\Hess^{(k-1,k)}_f(L^{\perp})$ is the matrix of the map $\mu_{L^{d-2k+1}}:A_k \to A_{d-k+1}$.    \\
Notice that the composition is the map $\mu_{L^{d-2k+2}}:A_{k-1}\to A_{d-(k-1)}$ and hence, by Theorem \ref{thm:generalization}, its matrix is just $\Hess_f^{k-1}(L^{\perp})$ whose determinant is non zero by hypothesis.
                                                                                                            
                                                 \end{proof}
                                          
A codimension $n$ Boolean $\K$-algebra can be presented as the complete intersection $$\K[x_1,\ldots,x_n]/(x_1^2, \ldots, x_n^2) \simeq \K[X_1,\ldots,X_n]/\operatorname{Ann}(x_1\ldots x_n).$$ It is a particular case of 
the algebras given by the annihilator of a monomial that have been treated by Stanley \cite{St} and Watanabe \cite{HMMNWW}. This result motivated the entire area and has been reproved by using different methods in \cite{RRR, Ik, HMMNWW}.
As a consequence of Proposition \ref{prop:inductiveconstruction} we give a simple proof that Boolean algebras have the SLP using Mixed Hessians. 

\begin{cor} \label{cor:booleanalgebras}
Let $\K$ be a field of characteristic zero. Then, the complete intersection algebra $\K[x_1,\ldots,x_n]/(x_1^2, \ldots, x_n^2)$ has the SLP. 
\end{cor}

\begin{proof}
 By induction in $n = \codim (A)$, the result is trivial for $n=1$. Suppose the result is true for a $n \geq 1$, then, for $f=x_1\ldots x_n$ all the $k$-th Hessians satisfy $\hess^k_f \neq 0$.
 Let us call $A = Q/\ann(f)$. To prove the result for $B = \K[x_1,\ldots,x_n,u]/(x_1^2, \ldots, x_n^2,u^2)$, we consider $g = uf = ux_1\ldots x_n \in \K[x_1,\ldots,x_n,u]$, 
 by Proposition \ref{prop:inductiveconstruction} the result follows.
 
\end{proof}

%\section{Algebras faling WLP}

\section{A combinatorial construction}

\begin{defin}\rm
 Let $V = \{u_1,\ldots,u_m\}$ be a finite set. A simplicial complex $\Delta$ with vertex set $V$ is a subset of the power set $ 2^V$, such that for all $A \in \Delta$ and for 
 all $B \subseteq A$ we have $B \in \Delta$. The members of $\Delta$ are referred as faces and maximal faces (with respect to the inclusion) are called facets. 
 If $A \in \Delta$ and $|A|=k$, it is called a $(k-1)$-face, or a face of dimension $k-1$. 
 If all the facets have the same dimension $d$ the complex is said to be homogeneous of (pure)
 dimension $d$. We say that $\Delta$ is a simplex if  $\Delta = 2^{V}$.
\end{defin}

In our context we identify the faces of a simplicial complex with monomials in the variables $\{u_1,\ldots,u_m\}$. Let $\K$ be any field and let $R = \K[u_1,\ldots,u_m]$ be the polynomial ring. To any finite subset $F \subset \{u_1,\ldots,u_m\}$ 
we associate the monomial $m_F = \displaystyle \prod_{u_i \in F}u_i $. In this way there is a natural bijection between the simplicial complex $\Delta$ and the set of the monomials $m_F,$ where $F$ a facet of $\Delta$.

Let $\Delta$ be a homogeneous simplicial complex of dimension $d-2$ whose facets are given by the monomials $g_i \in \K[u_1,\ldots,u_m]_{d-1}$. Let $f  \in \K[x_1,\ldots,x_n,u_1,\ldots,u_m]_{(1,d-1)}$ be the bihomogeneous form of monomial square free type
associated to $\Delta$, that is $f=f_{\Delta}=\displaystyle \sum_{i=1}^nx_ig_i$ (see Definition \ref{defin:bigraded1}).  The vertex set of $\Delta$ is also called $0$-skeleton and and we write $V=\{u_1,\ldots,u_m\}$. We identify the $1$-skeleton with 
a simple graph $\Delta_1=(V,E)$, hence the $1$-faces are called edges. Since, by differentiation, $X_i(f)=g_i$, we can identify each facet $g_i$ with the differential operator $X_i$. We denote by $e_k$ the number of $(k-1)$-faces, hence $e_1=m$ and $e_{d-1}=n$ and we put 
$e_0:=1$ and $e_j:=0$ for $j \geq d-1$. Let $A=Q/\ann_(f_{\Delta})$ be the associated algebra, we suppose that $I_1=0$. 

\begin{defin}\rm
 Let $\Delta$ be a homogeneous simplicial complex of dimension $d-2$. We will call $A_{\Delta} = Q/\ann(f_{\Delta})$ the associated algebra to $\Delta$.
\end{defin}

 \begin{defin}\rm
  Let $\Delta$ be a homogeneous simplicial complex of dimension $d-2$. We say that $\Delta $ is facet connected if for any pair of facets $F,F'$ of $\Delta$ there exists a sequence of facets, 
  $F_0=F,F_1,\ldots,F_s=F'$ such that $F_i \cap F_{i+1}$ is a $(d-3)$-face. We say that $\Delta$ is a flag complex if every collection of pairwise adjacent vertices spans a simplex.
   \end{defin}
 
\begin{thm}\cite{GZ}  \label{thm:mainpresentedbyquadrics}
Let $\Delta$ be a homogeneous simplicial complex of dimension $d-2\geq 1$ and let $A_{\Delta}$ be the associated Artinian Gorenstein algebra. $A$ is presented by quadrics if and only if $\Delta$ is a facet connected flag complex.
\end{thm}

\subsection{Turan algebras}

\begin{defin}\rm  Let $2 \leq a_1 \leq \ldots \leq a_{d-1}$ be integers. The Turan complex of order $a_1,\ldots,a_{d-1}$, $\Delta=\mathcal{TK}(a_1,\ldots,a_{d-1})$, is the homogeneous simplicial complex whose facet set is the 
Cartesian product $\pi = \displaystyle \prod_{i=1}^{d-1}\{1,2,\ldots,a_i\}$. The associated algebra is called the Turan algebra of order $(a_1,\ldots,a_{d-1})$ and denoted by $TA(a_1,\ldots,a_{d-1})$.
\end{defin}

\begin{thm} \cite{GZ} \label{thm:turanispresentedbyquadrics}
 Every Turan algebra $TA(a_1,\ldots,a_{d-1})$ is presented by quadrics. Its Hilbert vector is given by $h_k = s_{k-1}+s_{d-k-1}$ where $s_k = s_k(a_1,\ldots,a_{d-1})$ is the elementary symmetric polynomial of order $k$. 
\end{thm}

%\begin{defin}\rm  Let $2 \leq a_1 \leq \ldots \leq a_{d-1}$ be integers, the Turan complex of order $a_1,\ldots,a_{d-1}$, $\Delta=\mathcal{TK}(a_1,\ldots,a_{d-1})$, is the Homogeneous simplicial complex whose facets set is the 
%cartesian product $\pi = \displaystyle \prod_{i=1}^{d-1}\{1,2,\ldots,a_i\}$. The Turan polynomial of order $a_1,\ldots,a_{d-1}$ it the multihomogeneous polynomial 
%$f =f_{\Delta}= \displaystyle \sum_{\alpha \in \pi}x_{\alpha}u_{\alpha} \in R = \K[x_{\alpha},u_{(i,j_i)}]_{\alpha \in \pi,1\leq i\leq d-1,1\leq j_i \leq a_i},$ where $\alpha=(j_1,\ldots,j_{d-1})\in \pi$ and 
%$u_{\alpha} = u_{(1,j_1)}\ldots u_{(d-1,j_{d-1})}$. The Turan algebra of order $(a_1,\ldots,a_{d-1})$ is $TA(a_1,\ldots,a_{d-1}) = A_{\Delta}=Q/\ann_Q(f)$.
%\end{defin}

%\begin{rmk}\rm \label{rmk:symmetric} Notice that the number of $k$-faces of a Turan complex is $e_k=s_k$ where $s_k=s_k(a_1,\ldots,a_{d-1})$ is the symmetric function of order $k$.   
% By Theorem \ref{thm:mainpresentedbyquadrics}, the Hilbert vector of the Turan algebra $TA(a_1,\ldots,a_{d-1})$ is given by $h_k=s_{k-1}+s_{d-k-1}$ and every 
% Turan complex is a facet connected flag complex (see \cite[Theorem 3.7]{GZ}).
%\end{rmk}

\begin{lema}\label{lema:tecnicalhessian}
 Let $\Delta$ be a simplicial complex of pure dimension $d-2$ and let $A_{\Delta}=Q/\ann{f_{\Delta}}$ be the associated 
 algebra. Then the map $\mu_{L}: A_{k-1} \to A_k$, for $k \leq \frac{d}{2}$, is injective for a general $L \in A_1$ if, 
 and only if $\rk \Hess_f^{((1,k-2),(0,d-k))}= e_{d-k+1}$ and $\rk \Hess_f^{((0,k-1),(1,d-k-1))}=e_{k-1}$.
\end{lema}

\begin{proof} 
Since $A_k=A_{(1,k-1)}\oplus A_{(0,k)}$, and, since by Theorem \ref{thm:mainpresentedbyquadrics}, 
$\dim A_{(0,k)}=e_k$ and $$\dim A_{(1,k-1)}=\dim A_{(0,d-k)} = e_{d-k},$$
with a choice of bases 
consistent with the decomposition as direct sum, we have: 
$$\Hess_f^{(k-1,d-k)} = \left[ \begin{array}{cc} 
                                        0 & \Hess_f^{((1,k-2),(0,d-k))}\\
                                        \Hess_f^{((0,k-1),(1,d-k-1))} & \Hess_f^{((0,k-1),(0,d-k))}
                                       \end{array}
\right]_{(e_{d-k+1}+e_{k-1})\times (e_k+e_{d-k})},$$
where the matrices $\Hess_f^{((1,k-2),(0,d-k))}$ and $ \Hess_f^{((0,k-1),(1,d-k-1))}$ have order $e_{d-k+1}\times e_{d-k}$ and $e_{k-1}\times e_k$ respectively.\par

The injectivity of $\mu_L :A_{k-1}\to A_k$ implies $e_{d-k+1}+e_{k-1} \leq e_{d-k} +e_k$ and $\rk \Hess_f^{(k-1,d-k)} = e_{d-k+1}+e_{k-1}$. By the shape of the matrix 
this maximal rank can be achieved if and only if $\rk \Hess_f^{((1,k-2),(0,d-k))}= e_{d-k+1}$ and $\rk \Hess_f^{((0,k-1),(1,d-k-1))}=e_{k-1}$.\par

Conversely, if $\rk \Hess_f^{((1,k-2),(0,d-k))}= e_{d-k+1}$ and $\rk \Hess_f^{((0,k-1),(1,d-k-1))}=e_{k-1}$, then $\rk \Hess_f^{(k-1,d-k)} = e_{d-k+1}+e_{k-1}$ 
yielding the desired result.
\end{proof}

\begin{lema}\label{lema:turan222} Let $d \geq 3$ be an integer and consider the Turan complex $\mathcal{TK}(2^{(d-1)}):=\mathcal{TK}(2,\ldots,2) $ of dimension $d-1$. 
Let $f$ be the associated form. Then $$\rk \Hess_f^{((1,0),(0,d-2))}< 2^{d-1}.$$ 
\end{lema}

\begin{proof} Let us write $\Delta=\mathcal{TK}(2^{(d-1)})$. 
 First of all note that the rows of $\Hess_f^{((1,0),(0,d-2))}$ are indexed by the $2^{d-1}$ facets $x_{\alpha}$ of $\Delta$ and the columns are indexed by the $(d-2)$-faces $F$ of $\Delta$. 
 A non zero element of $\Hess_f^{((1,0),(0,d-2))}$ is a degree one monomial representing the remaining vertex of the facet $u_{\alpha}$ that does not belongs to the $(d-2)$-face $F$. 
 For instance, every column $F$ has only two non zero elements, say $u_{ij}$ and $u_{kj}$ representing the remaining vertex of the two faces that contain $F$. Furthermore, the other non zero elements 
 of the rows $i$ and $k$ are the same. \par
 If we multiply every row indexed by $x_{\alpha}$, with $\alpha=(j_1,\ldots,j_{d-1})$ where $j_i \in \{0,1\}$ by $(-1)^{j+1+\ldots+j_{d-1}}$, and by all the variables that do not 
 figure in the row. Then we get a matrix $M$ such that every column $j$ has only two non zero elements and they are opposite, say $M_j$ and $-M_j$. Summing up the rows, the result follows.  
\end{proof}

\begin{lema}\label{lema:subcomplex}
 Let $\Delta$ be a pure simplicial complex of dimension $d-2$ with $n$ facets and let $A_{\Delta}=Q/\ann{f_{\Delta}}$ be the associated algebra. 
 Let $v \in V(\Delta)$ be a vertex and denote $\Delta'= \Delta\setminus v$ 
 be the complex obtained from $\Delta$ by deleting $v$, let $n'$ be the number of facets of $\Delta'$ and let $A_{\Delta'}=Q/\ann{f_{\Delta'}}$ the associated algebra. Then
 $$\rk \Hess_f^{((1,0),(0,d-2))} = n \Rightarrow \rk \Hess_{f'}^{((1,0),(0,d-2))} = n'.$$
\end{lema}

\begin{proof}
 Let us choose an ordered basis of $A_{(0,1)}$ such that the last $n'$ vectors represent  the faces containing $v$. Let us choose a basis of $A_{(0,d-2)}$ in such a way that 
 the first vectors represent $d-2$ faces that does not contain $v$ and the last vectors the faces that contain $v$. The Matrix $\Hess_f^{((1,0),(0,d-2))}$ with respect to this basis is 
 $$\Hess_f^{((1,0),(0,d-2))} = \left[\begin{array}{cc}
                                \ast & \ast \\
                                0    & \Hess_{f'}^{((1,0),(0,d-2))}
                               \end{array}\right].
$$
The zero sub-matrix occurs by our choice of ordered basis. In fact, if $X_i$ represent a face not containing $v$, then $X_i(f)$ does not contain the variable $v$ and since 
the first vector of $A_{(0,d-2)}$ contain $v$, the derivative is zero. The result easily follows.
\end{proof}

\begin{defin}\rm
 Let $\Delta$ be a simplicial complex of pure dimension. We say that a new complex $\Delta'$ is constructed from $\Delta$ attaching a leaf if we add one vertex and one facet, that is, 
 $V_{\Delta'}=V_{\Delta} \cup \{v\}$ and $F_{\Delta'}=F_{\Delta} \cup \{F\}$ with $v \in F$. 
\end{defin}

The following Lemma will be useful in the sequel.

\begin{lema}\label{lema:attachcell}
 Let $\Delta$ be a $d-2 \geq 2$ dimensional simplicial complex and let $A_{\Delta}$ the associated algebra. Suppose that $A_{\Delta}$ is presented by quadrics and $e_1 \leq e_2$ and $e_{d-1}\leq e_{d-2}$. 
 Let $\Delta'$ be the simplicial complex constructed from $\Delta$ attaching a leaf. Then the algebra $A'=A_{\Delta'}$ associated to $\Delta'$ is presented by quadrics. 
 Moreover, if there is $L \in A_1$ such that $\mu_{L}:A_1 \to A_2$ is injective, then there is $L' \in A'$ such that $\mu_{L'}:A'_1 \to A'_2$ is also injective.
\end{lema}

\begin{proof}
 It is easy to see that $A'$ is presented by quadrics by Theorem \ref{thm:mainpresentedbyquadrics}. By Lemma \ref{lema:tecnicalhessian}, since $e_1 \leq e_2$ and $e_{d-1} \leq e_{d-2}$, 
 we have that $A$ satisfies the injective conjecture if, and only if $\rk \Hess_f^{((1,0),(0,d-2))}=e_{d-1}$ and $\rk \Hess_f^{((0,1),(1,d-3))}=e_{1}$. 
 Since attaching a leaf we still have $e'_1 \leq e'_2$ and $e'_{d-1} \leq e'_{d-2}$  and since it does not alter the fact that the rank of the desired mixed Hessians is maximal, the result follows.
\end{proof}

\begin{thm}\label{cor:matador}
 Let $A=TA(a_1,\ldots,a_{d-1})$ be the Turan algebra of order $(a_1,\ldots,a_{d-1})$ with $d \geq 3$ and $ 2 \leq a_1\leq a_2\leq\ldots \leq a_{d-1}$. 
 Then for all $L \in A_1$ the map $\mu_L : A_{1} \to A_2$ is not injective.
 %Furthermore, if $a_1 \approx \ldots \approx a_{d-1}$ are large enough, then 
 %$\Hilb(A)$ is not unimodal in the first step, that is, $\dim A_1 > \dim A_2$.
 \end{thm}

\begin{proof} 
 For $a_1=\ldots=a_{d-1}=2$, the result follows by Lemma \ref{lema:turan222} and Lemma \ref{lema:tecnicalhessian}. For 
 each $a_i>2$ we can delete vertex until obtain $\mathcal{TK}(2^{d-1})$ and by Lemma \ref{lema:subcomplex} and Lemma \ref{lema:tecnicalhessian} the result follows. \par
   
  %If $a_1 \approx \ldots \approx a_{d-1} \approx a$ are large enough, then, by a trivial Calculus I argument: 
  %$$\dim A_1 \approx (d-1)a+a^{d-1}>\binom{d-1}{2}a^2+(d-1)a^{d-2} \approx \dim A_2.$$  In this case, 
  %the Hilbert vector $\Hilb(A)$ is not unimodal in the first step.
 \end{proof}

%\section{Boolean algebras}

\section{Algebras presented by quadrics}

The WLP works in codimension $n \leq 2$, it is an open problem in codimension $n=3$ and there are algebras not satisfying it in codimension $n\geq 4 $.
Nevertheless, examples of Artinian algebras failing WLP were sporadic and the only systematic way to produce it were making the Hilbert vector non unimodal (see \cite{BI, Bo, BoL}).
In recent times the first author, in \cite{Go}, constructed families of algebras failing WLP. We recall the following result:

\begin{thm}\cite{Go} \label{thm:wlp}
 For each pair $(N,d)\not\in\{(3,3),(3,4),(4,4), (3,6)\}$ with $N \geq 3$ and with $d \geq 3$  there exist standard graded Artinian Gorenstein 
 algebras $A = \displaystyle  \oplus_{i=0}^d A_i$ of codimension $N+1$ and socle degree $d$, with a unimodal Hilbert vector that do not satisfy the WLP.
\end{thm}

On the other hand, for algebras presented by quadrics there was a conjecture posed in \cite{MN1, MN2}:

\begin{conjecture} {\bf (Migliore-Nagel WLP Conjecture)}\label{conj:SMNC}
 Any Artinian Gorenstein algebra presented by quadrics, over a field $\K$ of characteristic zero, has the Weak Lefschetz Property.
\end{conjecture}

The conjecture has been disproved by us in \cite[Cor. 3.8]{GZ}. In this section we study this phenomena in more details. We are looking for minimal examples of algebras presented by quadrics failing WLP.

\subsection{Artinian Gorenstein algebras with odd socle degree}

Let $A$ be a standard graded Artinian Gorenstein algebra with socle degree three, then $A=Q/\ann_Q(f)$ with $f \in R$ a homogeneous polynomial 
of degree $3$. Corollary \ref{cor:hessiancriteria} applied to this case tells us that $A$ satisfies the WLP if and only if $\hess_f \neq 0$.\par

By a result due to Dimca-Papadima, see see \cite[Thm. 1]{DP}, if $f$ is not a reduced polynomial and $\tilde{f}$ is its radical, then $\hess_f = 0$ if and only if $\hess_{\tilde{}f}=0$. 
For quadratic polynomials not defining a cone, $\hess_{\tilde{f}} \neq 0$, so we can restrict ourselves to reduced cubic polynomials. Furthermore, if $f=f_1f_2$ and $\hess_f=0$, then 
all the components of $X = V(f) \subset \P^n$ are developable, yielding $\hess_{f_i} \equiv 0 \pmod{f_i}$, in this case $f=l_1l_2l_3$ and $X$ is an arrangement of hyperplanes passing through a $\P^{N-2}$, which is a cone as soon as 
$N\geq 2$. 
So, from now on, we can restrict ourselves to the case that $f$ is an absolutely irreducible polynomial.

Let us recall Perazzo's construction which works like an atom for the constructions of forms with vanishing Hessian not defining a cone (see the Appendix of \cite{Go}).

\begin{defin}\rm
 A Perazzo polynomial is (up to a projective transformation) a form of type:
 $$f=\displaystyle \sum_{i=1}^s x_ig_i(\underline{u})+h(\underline{u})\in \K[\underline{x},\underline{u}]$$
 with $g_i \in \K[\underline{u}]_{d-1}$ linearly independent and algebraic dependent  and $h \in \K[\underline{u}]_{d}$.
\end{defin}

\begin{thm}\cite{Pe, GRu} \label{thm:perazzo} Perazzo hypersurfaces are not cones and have vanishing Hessian. 
Suppose that $N \leq 6$, and let $X = V(f) \subset \P^N$ be an irreducible cubic hypersurface which is not a cone and such that $\hess_f=0$. 
Then, up to a projective transformation, $f$ is a Perazzo polynomial.
\end{thm}

\begin{cor} \label{cor:cubics_low_codimension} Let $A$ be a standard graded Artinian Gorenstein $\K$-algebra of socle degree $3$. 
 If $A$ is presented by quadrics and $\codim A \leq 7$, then $A$ satisfies the WLP. 
\end{cor}

\begin{proof} Suppose that $A$ does not have the $WLP$. Then by the Hessian criterion, Corollary \ref{cor:hessiancriteria}, $\hess_f=0$. 
For $N \leq 6$, by Theorem \ref{thm:perazzo} and by \cite[Thm. 5.2,5.3,5.4]{GRu}, if $f$ is an irreducible cubic polynomial such that $\hess_f=0$, 
then, up to a projective transformation, either $f=xu_1^2+yu_1u_2+zu_2^2+h(\underline{u})$ or $N=6$ and $f = x_0g_0(\underline{u})+\ldots+x_3g_3(\underline{u})+h(\underline{u})$. 
Let us suppose, without loss of generality, that $x_4^2$ occurs as monomial only in $g_2$. In both cases, if we consider the associated standard graded Artinian Gorenstein 
one can verify directly that $X_2^3 \in \ann_Q(f)$ is a minimal generator.
\end{proof}

The next example was treated from the geometric point of view in \cite[p. 803, Example 6]{GRu}. It is the minimal counter-example for the MN-conjecture. 

\begin{ex}\rm \label{ex:semplicissimo}
 In $\P^7$ consider the cubic hypersurface $X=V(f) \subset \P^7$, given by $$f=\left|\begin{array}{ccc}
                                                                                      x_0 & x_1 & x_2\\
                                                                                      x_3 & x_4 & x_5\\
                                                                                      x_6 & x_7 & 0
                                                                                     \end{array}
\right| \in \K[x_0,\ldots,x_7].$$
As pointed out in \cite[p. 803, Example 6]{GRu}, $X$ represents a tangent section of the secant variety of the Segre variety $\operatorname{Seg}(\P^2 \times \P^2)\subset \P^8$. 
After a linear change of coordinates we can rewrite $f$ as a (Perazzo) bigraded polynomial of monomial square free type: $$f=x_1u_1u_2+x_2u_2u_3+x_3u_3u_4+x_4u_4u_1 \in R = \K[x_1,x_2,x_3,x_4,u_1,u_2,u_3,u_4].$$
Notice that $f_1f_3=u_1u_2u_3u_4=f_2f_4$, hence by the Gordan-Noether criterion, $\hess_f=0$. Let $A=Q/\ann_Q(f)$ be the associated algebra, of codimension $\codim A = 8$ and socle degree $3$. 
By the Hessian Lefschetz criterion, Theorem \ref{cor:hessiancriteria}, $A$ does not satisfy the WLP. 
On the other side, since its graph is a square, by Theorem \ref{thm:mainpresentedbyquadrics} it is presented by quadrics. Indeed, one can verify that  
$$
\begin{array}{c}
I= ({u}_{4}^{2},{u}_{2} {u}_{4},{x}_{2} {u}_{4},{x}_{1}
      {u}_{4},{u}_{3}^{2},{u}_{1} {u}_{3},{x}_{4} {u}_{3},{x}_{1}
      {u}_{3},{u}_{2}^{2},{x}_{4} {u}_{2},{x}_{3} {u}_{2},{x}_{2}
      {u}_{2}-{x}_{3} {u}_{4},{x}_{1} {u}_{2}-{x}_{4} {u}_{4},\\
      {u}_{1}^{2},{x}_{4} {u}_{1}-{x}_{3} {u}_{3},{x}_{3}
      {u}_{1},{x}_{2} {u}_{1},{x}_{1} {u}_{1}-{x}_{2}
      {u}_{3},{x}_{4}^{2},{x}_{3} {x}_{4},{x}_{2} {x}_{4},{x}_{1}
      {x}_{4},{x}_{3}^{2},{x}_{2} {x}_{3},{x}_{1} {x}_{3},{x}_{2}^{2},{x}_{1}
      {x}_{2},{x}_{1}^{2}).
 
\end{array}$$

\end{ex}

\begin{ex} \rm \label{ex:semplicissimo2} Consider the algebras $A=Q/\ann_Q(f)$ of codimension $r=9,11$.   
For $r=9$, take $f=x_1u_1u_2+x_2u_2u_3+x_3u_3u_4+x_4u_4u_1+w^2u_1$ and for $r=11$, take $f=x_1u_1u_2+x_2u_2u_3+x_3u_3u_4+x_4u_4u_1+x_5u_5u_1+w^2u_1$.
For both we have, $f_1f_3=f_2f_4$, hence by Gordan-Noether criterion and the Hessian criterion, $A$ does not have the WLP. 
One can verify that in all the cases $\ann_Q(f)$ is generated by quadrics. 
\end{ex}

\begin{lema}\label{lema:tree} Let $R=\K[u_1,\ldots,u_m]$ be the polynomial ring. Let $G=(V,E)$ be a connected graph such that $V= \{u_1,\ldots,u_m\}$ and 
$E$ is given by square free quadratic monomials. If $|E|=|V|$, then $G$ has a unique circuit $C$ and furthermore:
\begin{enumerate}
 \item If $|C|$ is even, then $\det \nabla G = 0$; 
 \item If $|C|$ is odd, then $\det \nabla G \neq 0$
\end{enumerate}
\end{lema}

\begin{proof}
 First of all let us show that the gradient matrix of any tree has maximal rank. Let $T=(V',E')$ be a tree where $V'=\{u_1,\ldots,u_m\}$ and $E'=\{g_1,\ldots,g_{m-1}\}$. 
 By induction on $m\geq 2$, the result is trivial for $m=2$. Let us suppose that for any tree with $|V'|=m \geq 2$, the gradient matrix has maximal rank. Let $\tilde{T}$ be a tree
 with $|\tilde{T}|=m+1$, $\tilde{T}=T \cup g_m$ where $g_m  = u_ju_{m+1}$, hence $$\nabla \tilde{T} = \left[\begin{array}{cc} 
                                                      \nabla T & \ast\\
                                                       0 & u_j
                                                      \end{array}\right] .
$$ The claim follows. \\
Let $T \subset G$ be a generating tree of $G$, then $T=(V,E')$ with $|E'|=|V|-1$, since $|V|=|E|$, $G$ contains a unique circuit, say 
$C=\{u_1u_2,u_2u_3,\ldots,u_{k-1}u_k,u_ku_1\}$ and let us suppose that $E=E'\cup u_ku_1$.

Since $G\setminus u_1u_k = T$ is a tree, $\nabla G = (\nabla T|\nabla g_m)$ where 
$$\nabla T =\left[ \begin{array}{ccccc}
              u_2 & 0 & \ldots & 0 & \ast \\
             u_1 & u_3 & \ldots & 0 & \ast\\
              0 & u_2 & \ldots & 0 & \ast\\
              \vdots & \vdots & \ddots & \vdots & \vdots \\
              0 & 0& \ldots & u_n & \ast \\
              0 & 0& \ldots & u_{n-1} & \ast \\
              0&0& \ldots & 0 & N
             \end{array} \right] \ \text{and}\ \ 
             \nabla G =\left[ \begin{array}{cccccc}
              u_2 & 0 & \ldots & 0 & \ast & u_k\\
              u_1 & u_3 & \ldots & 0 & \ast& 0\\
              0 & u_2 & \ldots & 0 & \ast& 0\\
              \vdots & \vdots & \ddots & \vdots & \vdots & \vdots \\
              0 & 0& \ldots & u_n & \ast & 0\\
              0 & 0& \ldots & u_{n-1} & \ast& u_1 \\
              0&0& \ldots & 0 & N & 0 \end{array} \right].
$$
Where in the last row $0$ represents $0_{(m-k,1)}$ and $N$ is a square matrix of order $m-k$ such that $\det(N)\neq 0$, since $\nabla T $ has maximal rank. 
Using the Laplace expansion, we get $$\det(\nabla G) = \det (N)u_1u_2\ldots u_n(1 + (-1)^{k-1}).$$ The result follows. 
\end{proof}

\begin{prop}\label{prop:maincubic} Let $f \in \K[x_1,\ldots,x_n,u_1,\ldots,u_m]_{(1,2)}$ be a bigraded cubic polynomial of monomial square free type and let $G=(E,V)$ be the associated 
graph. Then $A$ is presented by quadrics if, and only if, $G$ is connected and triangle free, in this case we have the following possibilities:
\begin{enumerate}
 \item If $G$ is a tree, then $A$ has the WLP;
 \item If $G$ contains only one circuit $|C|$, then either 
 \begin{enumerate}
  \item $|C|$ is even and $A$ does not have the WLP or
  \item $|C|\geq 5$ is odd and $A$ has the WLP
 \end{enumerate}
 \item If $G$ contains at least two circuits, then $A$ does not have the WLP.
\end{enumerate}

\end{prop}

\begin{proof} The characterization of the graphs that represent algebras presented by quadrics follows from Theorem \ref{thm:mainpresentedbyquadrics}.
We recall a very well known result that a set of $n$ monomials in $n$ variables are algebraically independent if and only if the incidence matrix has determinant different by zero. 
Since the incidence matrix of a graph of monomials is the gradient matrix evaluated in $u_1=1,\ldots,u_m=1$, the first and the second cases easily follow. 
In the last case $n> m$ hence the $g_i$ are algebraically dependent. The result follows from the Hessian criterion, Corollary \ref{cor:hessiancriteria} and the Gordan-Noether Theorem.
\end{proof}

\begin{cor}\label{cor:cubiccase} For  all $r \geq 8$ there exist standard graded Artinian Gorenstein $\K$-algebras of socle degree $3$ and $\codim A = r$, presented by quadrics, not satisfying the WLP. 
\end{cor}

\begin{proof} For the second one, if $r=9,11$ the result follows from Examples \ref{ex:semplicissimo} and \ref{ex:semplicissimo2}. \\
For all $r=2q \geq 8$ start with the square and then attach leaves as in Lemma \ref{lema:attachcell} and the result follows. 
For all $r=2q+1 \geq 13$ start with the hexagon together with the central diagonal and then attach leaves as in Lemma \ref{lema:attachcell} and the result follows. 
\end{proof}

\begin{lema}\label{prop:hessodd}
 Let $f \in \K[x_1,\ldots,x_r]$ be a homogeneous polynomial of degree $\deg(f) = 2k-1$ and such that $\hess^{k-1}_f = 0$ and let 
 $\tilde{f} = uvf \in \K[x_1,\ldots,x_n,u,v]$. Then $\hess^k_{\tilde{f}}=0$.   
\end{lema}

\begin{proof} Let $R=\K[x_1,\ldots,x_r]$ and $\tilde{R}=\K[x_1,\ldots,x_r,u,v]$ and let $Q$ and $\tilde{Q}$ be the associated rings of differential operators. 
 Let $A = Q/\ann_Q(f)$ and let $\tilde{A} = \tilde{Q}/\ann_{\tilde{Q}}(f)$. Applying twice Lemma \ref{lema:algebra}, we get
 $$\tilde{A}_k=A_{k-2}UV\oplus A_{k-1}U \oplus A_{k-1}V \oplus A_k.$$
 Therefore, 
 $$\Hess^k_{\tilde{f}} = \left[\begin{array}{cccc}
   0 &  0 & 0 & \ast \\
   0 & 0 & \Hess^{k-1}_f & \ast \\
   0 & \Hess^{k-1}_f  & 0 & \ast \\ 
   \ast & \ast & \ast & 0
   \end{array}\right].
$$

Since $\dim A_k = \dim A_{k-1}$, using Laplace's expansion on the second block row, it is easy to check that $\hess_f^{k-1}=0 \Rightarrow \hess^k_{\tilde{f}}=0$.

\end{proof}

\begin{cor}\label{cor:existemimpares}
 There exist standard graded Artinian Gorenstein algebras presented by quadrics of socle degree $d=2k+1\geq 3$ that do not satisfy the weak Lefschetz property when 
 $\codim A \geq d+5$. Furthermore, $A$ can be chosen to have Hilbert vector $\Hilb(A) = \sigma^{(d-3)}(1,r,r,1)$.
\end{cor}

\begin{proof}
 By Corollary \ref{cor:cubiccase}, for all $r \geq 8$ exist $f_r \in R=\K[x_1,\ldots,x_r]_3$ such that $\hess_{f_r}=0$ and $A_r$ is presented by quadrics. Let 
 $$\tilde{f}=\tilde{f}_{r,k}=f_ru_1\ldots u_{2k-2} \in \tilde{R}=\K[x_1,\ldots,x_r,u_1,\ldots,u_{2k-2}].$$
 We have $\deg \tilde{f} = 2k+1 \geq 3$. Let $\tilde{A}=\tilde{Q}/\ann_{\tilde{Q}}(\tilde{f})$. Then $$\codim \tilde{A} = 2k-2+r\geq 2k+6=d+5.$$ 
 By Lemma \ref{lema:ideal}, since $A_r$ is presented by quadrics, $\tilde{A}$ is also presented by quadrics. \par
 By Lemma \ref{lema:algebra} and by induction on $k$, we get that the Hilbert vector of $\tilde{A}$ are maximal.\par
 By induction on $k$ and Lemma \ref{prop:hessodd}, $\hess^k(\tilde{f})=0$, hence, by the Strong Lefschetz Hessian criterion,Corollary \ref{cor:hessiancriteria}, $\tilde{A}$ does not have the WLP.
\end{proof}

\subsection{Artinian Gorenstein algebras with even socle degree}

Notice that Lemma \ref{lema:algebra} and Lemma \ref{lema:ideal} together with the Hessian criterion 
Corollary \ref{cor:hessiancriteria} and the inductive procedure allow us to produce for any socle degree $d\geq 3$ and $\codim A \geq d+5$ Artinian Gorenstein algebras presented by quadrics that do not satisfy the SLP, but this construction, in even socle degree 
is not enough to failure of the WLP.\par

\begin{rmk}\rm For $f \in \K[x_1,\ldots,x_n,u_1,\ldots,u_m]_{(1,3)}$ be a quartic bihomogeneous polynomial of bidegree $(1,3)$ of monomial square free type. 
By Lemma \ref{lema:tree} and by Lemma \ref{lema:tecnicalhessian}, $\rk \Hess_f^{(1,2)}$ is maximal if and only if $\rk \Hess_f^{((1,0),(0,2))}=n$.
\end{rmk}

\begin{ex}\rm  Let $A$ be the algebra associated to the complex $\mathcal{TK}(2,2,3)\setminus e$ where $e$ is an edge having two incident faces, it has codimension $17$ and it does not have the WLP.
 Indeed, $|V|=7$, $|E|=15$ and $|F|=10$ but the matrix $\Hess_f^{((0,1),(2,0))}$ does not have maximal rank.
\end{ex}

\begin{cor}\label{cor:quarticcase}
 For all $r \geq 16$ there exist standard graded Artinian Gorenstein $\K$-algebras of socle degree $4$ and $\codim A = r$, presented by quadrics and not satisfying the WLP.
\end{cor}

\begin{proof}
 For codimension $r=2q \geq 14$, start with the Turan algebra $TA(2,2,2)$ of codimension $14$, which by Corollary \ref{cor:matador} does not have the WLP 
  and attach leaves as in Lemma \ref{lema:attachcell} and the result follows. For codimension $r=2q+1 \geq 17$, start with the Turan algebra $TA(2,2,3)$ of codimension $17$, which by Corollary \ref{cor:matador} does not have the WLP and attach leaves 
 to conclude the result. 
 \end{proof}

\begin{lema}\label{lema:truquepar}
 Let $A=Q/\ann_Q(f)$ be a standard graded Artinian Gorenstein algebra of socle degree $2q$ with $f \in \K[x_1,\ldots,x_n]$ and let $\tilde{f} \in \K[x_1,\ldots,x_n,u,v]$ 
 be $\tilde{f}=fuv$. If $\rk \Hess_f^{(q-1,q)}<\dim A_{q-1}$, then $\rk \Hess_{\tilde{f}}^{(q,q+1)}< \dim \tilde{A}_q$. 
\end{lema}

\begin{proof} Let $R=\K[x_1,\ldots,x_r]$ and $\tilde{R}=\K[x_1,\ldots,x_r,u,v]$ and let $Q,\tilde{Q},A, \tilde{A}$ as usual. Applying twice Lemma \ref{lema:algebra}, we get
 $$\tilde{A}_k=A_{k-2}UV\oplus A_{k-1}U \oplus A_{k-1}V \oplus A_k.$$
 Therefore, 
 $$\Hess^{(q,q+1)}_{\tilde{f}} = \left[\begin{array}{cccc}
   0 &  0 & 0 & \Hess^{(q-2,q+1)}_f \\
   0 & 0 & \Hess^{(q-1,q)}_f & V\Hess_f^{(q-1,q+1)} \\
   0 & \Hess^{(q-1,q)}_f  & 0 & U\Hess_f^{(q-1,q+1)} \\ 
   \Hess^{(q,q-1)}_f & V\Hess_f^q & U\Hess_f^q & 0
   \end{array}\right].
$$

Multiplying the second block row by $U$ and the third one by $-V$ and summing up we got a block row of type.

$$ \left[ \begin{array}{cccc}
   0  & -V\Hess^{(q-1,q)}_f & U \Hess^{(q-1,q)}_f & 0 
  \end{array} \right].
$$

Since, by hypothesis, $\rk \Hess^{(q-1,q)}_f $ is not maximal, the result follows.
\end{proof}

\begin{cor}\label{cor:existempares}
 There exists a standard graded Artinian Gorenstein algebras presented by quadrics of socle degree $d=2k\geq 4$ that does not have the weak Lefschetz property for each codimension $\codim A \geq d+12$. 
\end{cor}

\begin{proof}
 The result follows by induction in the same way as Corollary \ref{cor:existemimpares}.
\end{proof}

{\bf Acknowledgments}.
We wish to thank Francesco Russo and Junzo Watanabe for their suggestions leading to a significant improvement of the paper.

\end{document}